\documentclass[preprint,11pt,3p]{elsarticle}

\usepackage{times}
\usepackage{color}
\usepackage[colorlinks=true]{hyperref}

\usepackage{amsmath,amssymb,amsthm,amscd,mathrsfs}
\numberwithin{equation}{section}

\newtheorem{theorem}{Theorem}[section]

\newtheorem{proposition}{Proposition}[section]
\newtheorem{lemma}{Lemma}[section]

\newcommand{\al}{\alpha}     
\newcommand{\be}{\beta}      
\newcommand{\vare}{\varepsilon}
\newcommand{\lrnr}{\lfloor\rho n\rfloor}

\biboptions{sort&compress}

\journal{Nonlinearity}

\begin{document}
\allowdisplaybreaks

\begin{frontmatter}

\title{On the maximal run-length function in the L\"{u}roth expansion}
\author[label1]{Dingding Yu\corref{cor1}}
\ead{yudingding@hue.edu.cn}

\address[label1]{
School of Mathematics and Statistics,
Hubei University of Education, Wuhan, 430205, China}

\cortext[cor1]{corresponding author}

\begin{abstract}
Let \( \ell_n(x) \) denote the maximal run-length among the first \( n \) digits of the L\"{u}roth expansion of \( x\in(0,1] \). While \( \ell_n(x) \) grows logarithmically, we investigate the finer multifractal properties of the exceptional set where $\ell_n(x)$ exhibits linear growth. Specifically, we establish the Hausdorff dimension of the set  
\[
\left\{ x \in (0,1] : \liminf_{n \to \infty} \frac{\ell_n(x)}{n} = \alpha, \; \limsup_{n \to \infty} \frac{\ell_n(x)}{n} = \beta \right\},
\]
for all \( 0 \le \alpha \le \beta \le 1 \).
\end{abstract}
\begin{keyword}
the L\"{u}roth expansion\sep run-length function\sep Hausdorff dimension
\MSC[2010] 11K50\sep 28A80
\end{keyword}
\end{frontmatter}

\section{Introduction}


Let $T: (0, 1] \to (0, 1]$ be the L\"{u}roth transformation, defined by
\begin{equation}
    T(x) := \left\lfloor 1/x \right\rfloor\left(\left\lfloor 1/x \right\rfloor+1\right)\left(x - \frac{1}{\lfloor 1/x \rfloor + 1}\right),
\end{equation}
where $\lfloor \cdot \rfloor$ denotes the integer part function. For each $x \in (0,1]$, we define $d_1(x):=\lfloor 1/x \rfloor + 1$, and generally $d_n(x) = d_1(T^{n-1}(x))$ for $n \ge 2$. This leads to the L\"{u}roth expansion of $x$:
\begin{equation}\label{LE}
    x = \frac{1}{d_1(x)} + \sum_{n=2}^{\infty} \frac{1}{d_1(x)(d_1(x)-1)\cdots d_{n-1}(x)(d_{n-1}(x)-1)d_n(x)}.
\end{equation}
Here, $d_n(x) \ge 2$ are integers for any $n\ge 1$ and are called the expansion digits. We denote the expansion of $x$ by the sequence of digits $[d_1(x), d_2(x), \ldots]$. 

In this paper, we investigate the asymptotic behavior of maximal run-length function in the L\"{u}roth expansions. For any $x\in(0,1]$, the $n$-th maximal run-length function is defined as
$$\ell_n(x)=\max\left\{k\ge 1: d_{j+1}(x)=\cdots=d_{j+k}(x)\text{ for some }j,0\le j\le n-k\right\},$$
which represents the length of the longest block of identical digits among the first $n$ digits of the L\"{u}roth expansion of $x$. Sun and Xu \cite{sun2018maximal} established a law of large numbers for $\ell_n(x)$:
\begin{theorem}\cite[Theorem 1.1]{sun2018maximal}\label{lesun}
    For almost all $x\in(0,1]$, $\lim\limits_{n\to\infty}\ell_n(x)/\log_2 n=1$.
\end{theorem}
Furthermore, for any $0 \le \alpha \le \beta \le \infty$, the exceptional set
\[
\Bigl\{x \in (0,1] : \liminf_{n \to \infty} \frac{\ell_n(x)}{\log_2 n} = \alpha, \ \limsup_{n \to \infty} \frac{\ell_n(x)}{\log_2 n} = \beta\Bigr\}
\]
was shown to have full Hausdorff dimension in \cite{sun2018maximal}. Sun and Xu \cite{sun2018maximal} extended this result, let $\varphi:\mathbb{N}\to\mathbb{N}$ be monotonically increasing function satisfying $\lim\limits_{n\to\infty}\varphi(n)=\infty$ and $\liminf\limits_{n\to\infty}\varphi(n)/n=0$, the exceptional set
$$\left\{x\in(0,1]:\liminf_{n\to\infty}\frac{\ell_n(x)}{\varphi(n)}=\al, \limsup_{n\to\infty}\frac{\ell_n(x)}{\varphi(n)}=\be\right\}$$
is also of Hausdorff dimension 1.
In another direction, for a subset of digits $\mathcal{A}\subset \{2,3,\ldots,\}$, let $\ell_n(x,\mathcal{A})$ denote the length of the longest block whose elements all belong to $\mathcal{A}$. For a non-decreasing integer sequence $\{\varphi(n)\}_{n\ge 1}$, Zhou \cite{zhou2022dimension} determined the Hausdorff dimension of the set
$$\left\{x\in (0,1]: \limsup_{n\to\infty}\frac{\ell_n(x,\mathcal{A})}{\varphi (n)}=1\right\}.$$

Motivated by the aforementioned results, we extend this line of research by investigating the asymptotic growth rate of $\ell_n(x)$ relative to $n$. Specifically, for any $0\le \al\le \be\le 1$, we define the exceptional set
\begin{equation}\label{Eab}
    E(\alpha,\beta)=\left\{x\in(0,1]:\liminf_{n\to \infty}\frac{\ell_n(x)}{n}=\alpha,\limsup_{n\to\infty}\frac{\ell_n(x)}{n}=\beta\right\}.
\end{equation}
 We determine the Hausdorff dimension of $E(\alpha,\beta)$. Denote by $\dim_{\mathrm{H}}$ the Hausdorff dimension.

\begin{theorem}\label{mainthm}
    For any $u\ge 0$, denote by $s(u)$ the solution of the equation
    \begin{equation}\label{su}
        \sum_{t=2}^{\infty}\left(\frac{1}{2^ut(t-1)}\right)^s=1.
    \end{equation}
    Let $0\le \alpha\le \beta\le 1$ be two real numbers and  $E(\alpha,\beta)$ be defined as in \eqref{Eab}. Then
    \begin{equation*}
        \dim_{\mathrm{H}}E(\alpha,\beta)=
        \begin{cases}
         1,\ \ \ &\text{if}\ \ \be=0,\cr
         s\left(\frac{\be^2(1-\al)}{(1-\be)[\be-\al(1+\be)]}\right),\ \ \ &\text{if}\ \ 0\le \al<\frac{\be}{1+\be}<\be<1,\cr
        0,\ \ \ &\text{otherwise.}
        \end{cases}
    \end{equation*}
\end{theorem}

It should be mentioned that similar problems were first considered for the dyadic expansion by Erd\"{o}s and R\'{e}nyi \cite{er70}. Readers may consult the references for further information about run-length functions in dyadic expansions \cite{liwu16, liwu17, mawenwen07, sunxu17, zou11}, $\be$- expansions  \cite{hutongyu16, tongyuzhao16}, continued fraction expansions \cite{songzhou20, tanzhou25, wangwu11} and the power-$\alpha $-decaying Gauss-like expansions \cite{huangchen26}.


\section{Preliminaries}
This section is devoted to some elementary properties and dimensional results of the L\"{u}roth expansion that will be used later.
\begin{lemma}(\cite{dk02})
    The random variable sequence $\{d_n(x)\}_{n\ge 1}$ is independent and identically distributed with respect to the Lebesgue measure.
\end{lemma}

This independence property provides a probabilistic foundation for studying the metric theory of the L\"{u}roth expansion. Let $\Sigma=\{2,3,4,\ldots\}$. For any $n\in \mathbb{N}$ and $(d_1,\ldots,d_n)\in \Sigma^n$, we call
$$I_n(d_1,\ldots,d_n)=\{x\in(0,1]:d_i(x)=d_i~\text{for}~1\le i\le n\}$$
a cylinder of order $n$. The cylinder $I_n(d_1,\ldots,d_n)$ represents the set of the numbers in $(0,1]$, whose L\"{u}roth expansion beginning with $d_1,\ldots,d_n$.
\begin{lemma}(\cite{kf90})\label{lemmalength}
    For any $(d_1,\ldots,d_n)\in \Sigma^n, n\in \mathbb{N}, I_n(d_1,\ldots,d_n)$ is the interval with endpoints 
    $$\sum_{i=1}^{n}\frac{1}{d_1(d_1-1)\ldots d_{i-1}(d_{i-1}-1)d_i}$$
    and
    $$\sum_{i=1}^{n}\frac{1}{d_1(d_1-1)\ldots d_{i-1}(d_{i-1}-1)d_i}+\frac{1}{d_1(d_1-1)\ldots d_{n}(d_{n}-1)}.$$
    Thus
    $$|I_n(d_1,\ldots,d_n)|=\frac{1}{d_1(d_1-1)\ldots d_{n}(d_{n}-1)},$$
    where $|I|$ denotes the length of the interval $I$.
\end{lemma}

The following lemma gives the Hausdorff dimension of the exceptional set when the expansion digits are bounded, whose calculation relies on a Moran-type formula.
\begin{lemma}(\cite{sun2018maximal})
    For any $M\in \Sigma$, let
    $$E_M=\{x\in (0,1]:2\le d_i(x)\le M~~\text{for all}~~i\ge 1\}.$$
    Then $\dim_{\mathrm{H}}E_M=s_M$, where $s_M$ is given by the Moran formula
    $$\sum_{k=2}^{M}\left(\frac{1}{k(k-1)}\right)^{s_M}=1,$$
    and $\lim\limits_{M\to\infty}s_M=1.$
\end{lemma}

Furthermore, if the digits are bounded only on a set of indices of zero density, the corresponding set can attain the full dimension.
\begin{lemma}(\cite{sun2018maximal})
    Given a set of positive integers $\mathcal{J}=\{j_1<j_2<\cdots\}$ and an infinite bounded sequence $\{d_i\}_{i\ge 1}$ with $2\le d_i\le M$ for some $M\in \Sigma$, let
    $$E(\mathcal{J},\{d_i\})=\{x=[d_1(x),d_2(x),\ldots]\in(0,1]:d_i(x)=d_i,\forall i\in \mathcal{J}\}.$$
    If the density of $\mathcal{J}$ is zero, that is,
    $$\lim_{n\to\infty}\frac{\#\{i\le n:i\in\mathcal{J}\}}{n}=0,$$
    then $\dim_{\mathrm{H}}E(\mathcal{J},\{d_i\})=1$, where $\#$ denotes the number of elements in a set.
\end{lemma}

We now require some general tools for estimating the Hausdorff dimension.
\begin{lemma}(\cite{jg76})\label{lemma2.5}
    Let $E\subset \mathbb{R}$ and let $f:E\to\mathbb{R}$ be an $\eta$-H\"{o}lder function, i.e., for any $x,y\in E$,
    $$|f(x)-f(y)|\ll |x-y|^{\eta},$$
    where the constant in $\ll$ is an absolute constant. Then
    $$\dim_{\mathrm{H}}f(E)\le \frac{1}{\eta}\dim_{\mathrm{H}} E.$$
\end{lemma}

For establishing lower bounds on the Hausdorff dimension, the following mass distribution principle is essential.
\begin{lemma}(\cite{kf90})
    Let $E$ be a Borel set and $\mu$ be a measure with $\mu(E)>0$. Suppose that for some $s>0$, there exist constants $C>0$, $r_0>0$ such that for any $x\in E$ and $r<r_0$,
    $$\mu(B(x,r))\le Cr^s,$$
    where $B(x,r)$ denotes an open ball centered at $x$ and radius $r$. Then $\dim_{\mathrm{H}}E\ge s$.
\end{lemma}

Fix $u>0$ and an integer $M\in \Sigma$, let $s(u)$ and $s_M(u)$ be the unique roots of the equations \eqref{su} and
\begin{equation}\label{suM}
    \sum_{t=2}^{M}\left(\frac{1}{2^ut(t-1)}\right)^s=1,
\end{equation}
respectively. The lemma below exhibits some basic properties shared by $s(u)$ and $s_M(u)$.

\begin{lemma}\label{continus}
    Let $M\in \Sigma$ and $u\ge0$, we have
    \begin{itemize}
        \item[(1)] $s(u)$ and $s_M(u)$ are both non-increasing and continuous with respect to $u$;
        \item[(2)]$s(0)=1$ and $\lim\limits_{u\to\infty}s(u)=0$;
        \item[(3)] the limit $\lim\limits_{M\to\infty}s_M(u)$ exists, and is equal to $s(u)$.
    \end{itemize}
\end{lemma}

\begin{proof}
Denote $a_t = 2^{-u}/[t(t-1)]$ for $t\ge 2$, and define
\[
F_u(s) = \sum_{t=2}^\infty a_t^{s}, \qquad
F_{M,u}(s) = \sum_{t=2}^M a_t^{s}.
\]
For fixed $u$, each $a_t^{s}$ is strictly decreasing in $s$, hence $F_u(s)$ and $F_{M,u}(s)$ are strictly decreasing in $s$. Since $F_u(0+)=+\infty$, $F_u(\infty)=0$, and similarly for $F_{M,u}$ (with $F_{M,u}(0)=M-1\ge1$), the equations $F_u(s)=1$ and $F_{M,u}(s)=1$ have unique positive solutions $s(u)$ and $s_M(u)$ respectively.

\begin{itemize}
    \item[(1)] For fixed $s$, $a_t$ is strictly decreasing in $u$, so $F_u(s)$ is strictly decreasing in $u$. If $u_1<u_2$, suppose $s(u_1)\le s(u_2)$. Then
    \[
    1 = F_{u_1}(s(u_1)) \ge F_{u_1}(s(u_2)) > F_{u_2}(s(u_2)) = 1,
    \]
    a contradiction. Hence $s(u_1) > s(u_2)$ and $s(u)$ is strictly decreasing in $u$. The same argument works for $s_M(u)$. Continuity follows from the implicit function theorem.
    \item[(2)] When $u=0$, $a_t = 1/[t(t-1)]$. Direct calculation shows $\sum_{t=2}^\infty a_t = 1$, so $s(0)=1$. For $s<1$, $F_0(s) > 1$; for $s>1$, $F_0(s) < 1$; hence $s(0)=1$ uniquely.
    As $u\to\infty$, for any fixed $s>0$, $a_t^{s}$ tends to $ 0$, so $F_u(s)\to 0$. Thus to satisfy $F_u(s)=1$, $s(u)$ must tend to $0$.
    \item[(3)] For each fixed $u$, $F_{M,u}(s)$ tends to $F_u(s)$ as $M\to\infty$, monotonically in $M$. Let $s^* = \limsup\limits_{M\to\infty} s_M(u)$. If $s^* > s(u)$, take a subsequence $s_{M_k}\to s^*$, for $k$ large enough, $s_{M_k} > s(u)+\varepsilon$, then
    \[
   1 = F_{M_k,u}(s_{M_k}) < F_{M_k,u}(s(u)+\varepsilon) \le F_u(s(u)+\varepsilon) < F_u(s(u)) = 1,
   \]
   a contradiction. If $s^* < s(u)$, take subsequence $s_{M_k}\to s^* < s(u)-\varepsilon$, then
   \[
   1 = F_{M_k,u}(s_{M_k}) > F_{M_k,u}(s(u)-\varepsilon) \to F_u(s(u)-\varepsilon) > 1,
   \]
  also a contradiction. Hence $\lim_{M\to\infty} s_M(u) = s(u)$.
  \end{itemize}
\end{proof}

The following elementary estimate will be used repeatedly in the calculations of the Hausdorff measure.
\begin{lemma}\label{prop:main-sum-estimate}
Let $s > 0$ and $L \ge 1$ be an integer. Then there exists a constant $C > 0$ (independent of $L$) such that
\[
\sum_{i=2}^{\infty} \left( \frac{1}{i(i-1)} \right)^{sL} \le C \left( \frac12 \right)^{sL}.
\]
In particular, if $s = s(u)+\varepsilon$ for some $u \ge 0$ and $\varepsilon > 0$, where $s(u)$ is defined by \eqref{su}, then the constant $C$ can be chosen to depend only on $\varepsilon$ (and not on $u$ or $L$).
\end{lemma}

\begin{proof}
Denote $a_i = \frac{1}{i(i-1)}$ for $i \ge 2$.  
Observe that $a_2 = \frac12$, and for $i \ge 3$ we have $i(i-1) \ge \frac12 i^2$, hence, for $i\ge 3$,
$a_i \le \frac{2}{i^2}.$
Split the sum as 
\[
\sum_{i=2}^\infty a_i^{sL} = a_2^{sL} + \sum_{i=3}^\infty a_i^{sL}.
\]
The first term equals $(1/2)^{sL}$.  
For the tail, using the inequality
\[
\sum_{i=3}^\infty a_i^{sL} \le \sum_{i=3}^\infty \left(\frac{2}{i^2}\right)^{sL}
= 2^{sL} \sum_{i=3}^\infty i^{-2sL}.
\]

Since $s>0$ and $L\ge 1$, we have $2sL > 1$ for sufficiently large $L$. For the finitely many smaller $L$, the sum is trivially bounded by a constant depending only on $s$.  Assume now $2sL > 1$.  Then
\[
\sum_{i=3}^\infty i^{-2sL} \le \int_2^\infty x^{-2sL}\, dx 
= \frac{2^{1-2sL}}{2sL-1}.
\]
Consequently,
\[
\sum_{i=3}^\infty a_i^{sL} \le 2^{sL}\cdot\frac{2^{1-2sL}}{2sL-1}
= \frac{2^{1-sL}}{2sL-1}.
\]
Combining the two parts,
\[
\sum_{i=2}^\infty a_i^{sL} \le \left(\frac12\right)^{sL} + \frac{2^{1-sL}}{2sL-1}
= \left(\frac12\right)^{sL}\Bigl[\,1 + \frac{2}{2sL-1}\,\Bigr].
\]

If $s = s(u) + \varepsilon$ with $\varepsilon > 0$, then $sL \ge \varepsilon L$.  
Choose $L_0$ such that $2\varepsilon L_0 > 1$.  For all $L \ge L_0$, we have
\[
1 + \frac{2}{2sL-1} \le 1 + \frac{2}{2\varepsilon L_0 - 1}.
\]
For $L < L_0$ (only finitely many values), the left–hand side of the original inequality is bounded by a constant depending only on $\varepsilon$ (and on the fixed $L_0$).  
Thus there exists a constant $C = C(\varepsilon)$ such that
\[
\sum_{i=2}^\infty a_i^{sL} \le C \left(\frac12\right)^{sL}
\qquad \text{for all } L \ge 1.
\]
This completes the proof.
\end{proof}

\section{Proof of Theorem \ref{mainthm}: Upper Bound}\label{sec:upper-bound}

We discuss case by case according to the range of the parameters $(\alpha, \beta)$, where $0 \le \alpha \le \beta \le 1$.

\subsection{\bf{Case 1: }$\beta = 0$}\label{case:beta0}

Since $\beta = 0$, the condition implies $\alpha = 0$. By Theorem \ref{lesun}, we note that $\lim_{n\to\infty}\ell_n(x)/\log_2 n =1$ for almost all $x\in (0,1]$, hence, $\mathcal{L}(E(0,0))=1$. Therefore, 
$$\dim_{\mathrm{H}}E(0,0)=1.$$

\subsection{\bf{Case 2:} $\beta = 1$}
\label{casebeta1}

Here $0 \le \alpha \le 1$. By the definition of $E(\al,1)$, we have
$$\limsup_{n\to\infty}\frac{\ell_n(x)}{n}=1.$$
Then, for any $0<\rho<1$ there exist infinity many $n$ such that $\ell_n(x)/n>\rho$. Therefore
\begin{align*}
    E(\al,1)&\subset\{x\in(0,1]:\ell_n(x)>\rho n~~\text{for infinity many}~~n\}\\ &\subset\bigcap_{N=1}^{\infty}\bigcup_{n=N}^{\infty}\bigcup_{i=2}^{\infty}\bigcup_{j=0}^{n-\lfloor\rho n\rfloor}\{x\in(0,1]:d_{j+1}(x)=\cdots=d_{j+\lfloor \rho n\rfloor}(x)=i\}\\ &\subset\bigcap_{N=1}^{\infty}\bigcup_{n=N}^{\infty}\bigcup_{i=2}^{\infty}\bigcup_{j=0}^{n-\lfloor\rho n\rfloor}\bigcup_{d_1,\ldots,d_j\in \Sigma\atop d_{j+\lfloor\rho n\rfloor+1},\ldots,d_n\in\Sigma}I_n(d_1,\ldots,d_j,i,\ldots,i,d_{j+\lfloor\rho n\rfloor+1},\ldots,d_n).
\end{align*}

For any $\vare>0$, there exists positive $N_0$ such that for all $n\ge N_0$, we have 
\begin{equation}\label{rho1}
   \frac{\lfloor\rho n\rfloor}{n-\lfloor\rho n\rfloor}>\frac{\rho}{1-\rho}-\vare. 
\end{equation}
Write $s=s(\frac{\rho}{1-\rho})+\vare$,
then the $s$-dimensional Hausdorff measure of $E(\al,1)$ can be exhibited as follows. Let $\mathcal{H}^s$ denote the $s$-dimensional Hausdorff measure.
\begin{align*}
    &\mathcal{H}^s(E(\al,1))\\
    &\le \liminf_{N\to\infty}\sum_{n=N}^{\infty}\sum_{i=2}^{\infty}\sum_{j=0}^{n-\lrnr}\sum_{d_1,\ldots,d_j\in \Sigma\atop d_{j+\lfloor\rho n\rfloor+1},\ldots,d_n\in\Sigma}|I_n(d_1,\ldots,d_j,i,\ldots,i,d_{j+\lfloor\rho n\rfloor+1},\ldots,d_n)|^s\\
    &\le\liminf_{N\to\infty}\sum_{n=N}^{\infty}\sum_{i=2}^{\infty}n\left(\frac{1}{i(i-1)}\right)^{s\lrnr}\left[\sum_{t=2}^{\infty}\left(\frac{1}{t(t-1)}\right)^s\right]^{n-\lrnr}\\
    & \le\liminf_{N\to\infty}\sum_{n=N}^{\infty}\sum_{i=2}^{\infty}Cn\left(\frac{1}{2}\right)^{n\vare}\left(\frac{1}{i(i-1)}\right)^{s(\frac{\rho}{1-\rho})\lrnr}\left[\sum_{t=2}^{\infty}\left(\frac{1}{t(t-1)}\right)^{s(\frac{\rho}{1-\rho})}\right]^{n-\lrnr}\\
    &
    \le\liminf_{N\to\infty}\sum_{n=N}^{\infty}Cn\left(\frac{1}{2}\right)^{n\vare}\left(\frac{1}{2}\right)^{s(\frac{\rho}{1-\rho})\lrnr}\left[\sum_{t=2}^{\infty}\left(\frac{1}{t(t-1)}\right)^{s(\frac{\rho}{1-\rho})}\right]^{n-\lrnr}\\
    &
    \le\liminf_{N\to\infty}\sum_{n=N}^{\infty}Cn\left(\frac{1}{2}\right)^{\vare n+s(\frac{\rho}{1-\rho})(\lrnr-\frac{\rho}{1-\rho}(n-\lrnr))}\left[\left(\frac{1}{2}\right)^{\frac{\rho}{1-\rho}s(\frac{\rho}{1-\rho})}\sum_{t=2}^{\infty}\left(\frac{1}{t(t-1)}\right)^{s(\frac{\rho}{1-\rho})}\right]^{n-\lrnr}\\
    &\le\liminf_{N\to\infty}\sum_{n=N}^{\infty}Cn \left(\frac{1}{2}\right)^{\vare(n-s(\frac{\rho}{1-\rho})(n-\lrnr))}<+\infty,
\end{align*}
where the third inequality comes true because of the Lemma \ref{prop:main-sum-estimate} and the penultimate inequality comes from \eqref{rho1} and \eqref{su}.
Since $\vare$ is arbitrary, we have $\dim_{\mathrm{H}}E(\al,1)\le s\left(\frac{\rho}{1-\rho}\right)$. Letting $\rho\to\be=1$, we obtain
$$\dim_{\mathrm{H}}E(\al,1)=0.$$

\subsection{\bf{Case 3:} $0 < \beta < 1$}
\label{case:beta-intermediate}

This is the most involved situation. We refine the discussion according to the value of $\alpha$.

\subsubsection{\bf{Case 3.1:} $\alpha = 0$}
\label{case:beta<1-alpha0}

The upper bound estimation on $\dim_{\mathrm{H}}E(0,\be)$ can be established in much the same way as Subsection \ref{casebeta1}. Thus, in this case
$$\dim_{\mathrm{H}}E(0,\be)\le s\left(\frac{\be}{1-\be}\right).$$

\subsubsection{\bf{Case 3.2:} $0 < \alpha \le \beta<1$}
\label{case:beta<1-alpha>0}

Write $$P=\{x\in(0,1]:\text{the L\"{u}roth expansion of $x$ is periodic}\}.$$
The set $P$ can be expressed as
$$P=\bigcup_{k=1}^{\infty}\bigcup_{d_1=2}^{\infty}\cdots\bigcup_{d_j=2}^{\infty}\bigcup_{(d_{j+1},\ldots,d_{j+k})\in\Sigma^k}\{[d_1,\ldots,d_j,\overline{d_{j+1},\ldots,d_{j+k}}]\}$$
and thus it is countable.

Fix $x\in E(\al,\be)\cap ((0,1]\backslash P)$. We define two sequences of positive integers $\{n'_k\}_{k\ge 1}$ and $\{m'_k\}_{k\ge 1}$ as follows.
\begin{align*}
    n'_0&=0,\qquad m'_0=0,\\
    n'_k&=\min\{n\ge n'_{k-1}+m'_{k-1}:d_{n+1}(x)=d_{n+2}(x)\},\\
    m'_k&=\max\{j\ge 2:d_{n'_k+1}(x)=\ldots=d_{n'_k+j}(x)\}.
\end{align*}
Since $\limsup\limits_{n\to\infty}\frac{\ell_n(x)}{n}=\be>0$, there exist infinitely many $n$ such that $\ell_n(x)>\be n$, this implies the existence of run-length at least $2$ starting at some position. Thus the sequences $\{n'_k\}$ and $\{m'_k\}$ are well-defined and infinite. Moreover, $\beta>0$ implies $\lim\limits_{k\to\infty}m'_k=+\infty$, so we can extract a strictly increasing subsequence of $\{m'_k\}$ as follows. Let
$$j_1=1,\qquad j_{k+1}=\min\{j>j_k:m'_j>m'_{j_k}\},$$
and define $m_k=m'_{j_k}$, $n_k=n'_{j_k}$. By construction, $\{m_k\}$ is strictly increasing to infinity and 
\begin{equation}\label{mk}
    2\le m_k\le n_{k+1}-n_k.
\end{equation}

Assume that $n_k+m_k\le n<n_{k+1}+m_{k+1}$ for some $k$. Then the maximal run-length within the first $n$ digits is:
\begin{equation*}
    \ell_n(x)=
    \begin{cases}
         m_k,\ \ \ &\text{if}\ \ n_k+m_k\le n\le n_{k+1}+m_k,\cr
         n-n_{k+1},\ \ \ &\text{if}\ \ n_{k+1}+m_k< n<n_{k+1}+m_{k+1}.
    \end{cases}
\end{equation*}
From this, we deduce that
\begin{align}
    \al&=\liminf_{n\to\infty}\frac{\ell_n(x)}{n}=\liminf_{k\to\infty}\frac{m_k}{n_{k+1}+m_k},\label{al1}\\ 
    \be&=\limsup_{n\to\infty}\frac{\ell_n(x)}{n}=\limsup_{k\to\infty}\frac{m_k}{n_{k}+m_k}.
    \label{be1}
\end{align}
From \eqref{be1} we also obtain
\begin{equation}\label{be2}
    \limsup_{k\to\infty}\frac{\ell_{n_k+2m_k}(x)}{n_k+2m_k}=\limsup_{k\to\infty}\frac{m_k}{n_k+2m_k}=\frac{\be}{1+\be},
\end{equation}
since $m_k/(n_k+m_k) \to \beta$ along a subsequence implies $m_k/n_k \to \beta/(1-\beta)$.

By \eqref{mk}, \eqref{al1} and \eqref{be2}, we have
$$\al\le \frac{m_k}{n_{k+1}+m_k}\le \frac{m_k}{n_k+2m_k}\le \frac{\be}{1+\be}.$$
Therefore, for any $x\in E(\al,\be)\cap((0,1]\backslash P)$, we must have $\al\le \frac{\be}{1+\be}$.

To obtain a sharp upper bound, we compare $\alpha$ with $\dfrac{\beta}{1+\beta}$.

\begin{description}
    \item[Case 3.2 (1): $0 < \dfrac{\beta}{1+\beta}<\alpha \le\be<1$.] 
    If $\al>\frac{\be}{1+\be}$, then $E(\al,\be)\subset P$. Therefore, in this case, $E(\al,\be)$ is at most countable and 
    $$\dim_{\mathrm{H}}E(\al,\be)=0.$$
    
    \item[Case 3.2 (2): $0< \alpha\le\dfrac{\beta}{1+\beta}  < \beta<1$.] 

    From \eqref{al1} and \eqref{be1} we derive that
\begin{align}
    &\limsup_{k\to\infty}\frac{n_{k+1}}{m_k}=\frac{1-\al}{\al}, \label{al2}\\
    &\limsup_{k\to\infty}\frac{m_{k+1}}{n_{k+1}}=\limsup_{k\to\infty}\frac{m_k}{n_k}=\frac{\be}{1-\be}, \label{be3}\\
    &\limsup_{k\to\infty}\frac{m_{k+1}}{m_k}= \limsup_{k\to\infty}\frac{m_{k+1}}{n_{k+1}}\cdot\frac{n_{k+1}}{m_k}\le\frac{\be(1-\al)}{\al(1-\be)}.\label{albe1}
\end{align}
Now, from \eqref{al2}, \eqref{be3} and \eqref{albe1}, one can estimate
\begin{align}
    \limsup_{k\to\infty}\frac{\sum_{j=1}^k m_j}{n_k+m_k}
    &\ge \liminf_{k\to\infty}\frac{\sum_{j=1}^{k-1} m_j}{n_k+m_k}+\limsup_{k\to\infty}\frac{m_k}{n_k+m_k}\nonumber\\
    &\ge \liminf_{k\to\infty}\frac{m_k}{n_{k+1}+m_{k+1}-(n_k+m_k)}+\limsup_{k\to\infty}\frac{m_k}{n_k+m_k}\nonumber\\
    &\ge\frac{\be^2(1-\al)}{\be-\al}.\label{albe2}
\end{align}
At the same time, we can deduce that, the sequence $\{n_k+m_k\}_{k\ge 1}$ grows at least exponentially, that is, there exists a constant $C_1>0$ such that 
$$\log (n_k+m_k)\ge k C_1,$$
for all sufficiently large $k$.
\\ \hspace*{1em} 
Write $\eta=\frac{\be^2(1-\al)}{\be-\al}$. Then by \eqref{albe2}, we have 
$$\limsup_{k\to\infty}\frac{\sum_{j=1}^km_j}{n_k+m_k-\sum_{j=1}^km_j}\ge \frac{\eta}{1-\eta}.$$

Therefore,
\begin{itemize}
    \item If $0<\al<\frac{\be}{1+\be}<\be<1$, then for any $\vare>0$, there exist infinitely many $k$ such that 
    \begin{equation}\label{mk1}
        \sum_{j=1}^km_j\ge\left(\frac{\eta}{1-\eta}-\vare\right)\left(n_k+m_k-\sum_{j=1}^km_j\right).
    \end{equation}
    \item If $0<\al=\frac{\be}{1+\be}<\be<1$, then for any $M>0$ large enough, there exist infinitely many $k$ such that
    \begin{equation}\label{mk2}
        \sum_{j=1}^km_j\ge M\left(n_k+m_k-\sum_{j=1}^km_j\right).
    \end{equation} 
\end{itemize}
Recall that the construction of $\{n_k\}_{k\ge 1}$ and $\{m_k\}_{k\ge 1}$ depends on $x\in E(\al,\be)$.
\\ \hspace*{1em} 
When $0<\al<\frac{\be}{1+\be}<\be<1$, we define
$$\mathcal{A}=\{(\{n_k\},\{m_k\}): \text{the relations \eqref{mk}-\eqref{be1} hold}\},$$
which is the collection of all sequence pairs that can arise from points in $E(\al,\be)$. For $(\{n_k\},\{m_k\})\in \mathcal{A}$ and a sequence $\{i_k\}_{k\ge 1}$ with $i_k\in \Sigma$, set
\begin{align*}
    &\mathcal{W}_k=\{x\in(0,1]:d_{n_k+1}(x)=\cdots=d_{n_k+m_k}(x)~\text{for all}~k\},\\
    &\mathcal{C}_k=\left\{((n_j,m_j)_{j=1}^{k-1},n_k):
    \begin{array}{l}
        2\le m_j<n_{j+1}-n_j~\text{for }1\le j<k,\ m_k\ge 2,\\
        \text{and condition \eqref{mk1} holds}
    \end{array}
    \right\},\\
    &\mathcal{O}_k(\{i_j\}_{j=1}^k)=\{(\tau_1,\ldots,\tau_{n_k+m_k})\in \Sigma^{{n_k+m_k}}: \tau_{n_j+1}=\cdots=\tau_{n_j+m_j}=i_j \text{ for }1\le j\le k\}.
\end{align*}
Then we have the covering
\begin{equation}\label{Ecover}
    E(\al,\be)\subset \bigcup_{(\{n_k\},\{m_k\})\in \mathcal{A}}\mathcal{W}_k\subset \bigcap_{m=1}^{\infty}\bigcup_{k=m}^{\infty} \Gamma_k,
\end{equation}
where 
$$\Gamma_k=\!\!\!\bigcup_{n_k+m_k\ge e^{kC_1}}^{\infty}\ 
\bigcup_{((n_j,m_j)_{j=1}^{k-1},n_k)\in\mathcal{C}_k}\ 
\bigcup_{\substack{i_1,\ldots,i_k\in\Sigma \\ (d_1,\ldots,d_{n_k+m_k})\in\mathcal{O}_k(\{i_j\}_{j=1}^k)}}\!\!\! I_{n_k+m_k}(d_1,\ldots,d_{n_k+m_k}).$$
\\ \hspace*{1em} 
Fix $\vare>0$, and set $\tilde{s} = s\!\left(\frac{\eta}{1-\eta}-\vare\right)$. 
For sufficiently large $k$, we have
\begin{equation}\label{estimate2}
    k^{\frac{2}{C_1}\log k} \le 2^{k\vare}.
\end{equation}

Now, let us estimate the $(\tilde{s}+2\vare)$-dimensional Hausdorff content of $\Gamma_k$:
\begin{align}\label{cover1}
    v_k(\tilde{s}&+2\vare):=\\
    &\nonumber\!\!\!\sum_{n_k+m_k\ge e^{kC_1}}^{\infty}\ 
\sum_{((n_j,m_j)_{j=1}^{k-1},n_k)\in\mathcal{C}_k}\sum_{\substack{i_1,\ldots,i_k\in\Sigma \\ (d_1,\ldots,d_{n_k+m_k})\in\mathcal{O}_k(\{i_j\}_{j=1}^k)}}\!\!\! |I_{n_k+m_k}(d_1,\ldots,d_{n_k+m_k})|^{\tilde{s}+2\vare},
\end{align}
for all large $k$.
\\ \hspace*{1em} 
The innermost sum can be bounded as follows:
\begin{align}
    &\sum_{i_1,\ldots,i_k\in\Sigma \atop (d_1,\ldots,d_{n_k+m_k})\in\mathcal{O}_k(\{i_j\}_{j=1}^k)}\nonumber|I_{n_k+m_k}(d_1,\ldots,d_{n_k+m_k})|^{\tilde{s}+2\vare}\\
    &\nonumber
    \le\sum_{i_1,\ldots,i_k\in\Sigma }\sum_{ (d_1,\ldots,d_{n_k+m_k})\in\mathcal{O}_k(\{i_j\}_{j=1}^k)}\left(\frac{1}{d_1(d_1-1)\cdots d_{n_k+m_k}(d_{n_k+m_k}-1)}\right)^{\tilde{s}+2\vare}\\
    &\nonumber
    \le\left(\frac{1}{2}\right)^{2(n_k+m_k)\vare}\sum_{i_1,\ldots,i_k\in\Sigma }\sum_{ (d_1,\ldots,d_{n_k+m_k})\in\mathcal{O}_k(\{i_j\}_{j=1}^k)}\left(\frac{1}{d_1(d_1-1)\cdots d_{n_k+m_k}(d_{n_k+m_k}-1)}\right)^{\tilde{s}}\\
    &\nonumber
    =\left(\frac{1}{2}\right)^{2(n_k+m_k)\vare}\left[\sum_{t=2}^{\infty}\left(\frac{1}{t(t-1)}\right)^{\tilde{s}}\right]^{n_k+m_k-\sum_{j=1}^km_j}\prod_{j=1}^k\left[\sum_{i=2}^{\infty}\left(\frac{1}{i(i-1)}\right)^{m_j\tilde{s}}\right]\\
    &\nonumber
    \le\left(\frac{1}{2}\right)^{2(n_k+m_k)\vare}C^k\left(\frac{1}{2}\right)^{\tilde{s}\sum_{j=1}^km_j}\left[\sum_{t=2}^{\infty}\left(\frac{1}{t(t-1)}\right)^{\tilde{s}}\right]^{n_k+m_k-\sum_{j=1}^km_j} \\
    &\nonumber
    \le\left(\frac{1}{2}\right)^{2(n_k+m_k)\vare}C^k\left[\left(\frac{1}{2}\right)^{(\frac{\eta}{1-\eta}-\vare)\tilde{s}}\sum_{t=2}^{\infty}\left(\frac{1}{t(t-1)}\right)^{\tilde{s}}\right]^{n_k+m_k-\sum_{j=1}^km_j}\\
    &=\left(\frac{1}{2}\right)^{2(n_k+m_k)\vare}C^k,\label{etimate3}
\end{align}
where the penultimate and last inequalities are respectively from the Lemma \ref{prop:main-sum-estimate} and \eqref{mk1}.
Next, we bound the cardinality of $\mathcal{C}_k$. For each $k$, the number of integer pairs $(n_j,m_j)_{j=1}^{k-1}$ and $n_k$ satisfying the conditions in $\mathcal{C}_k$ is at most $(n_k+m_k)^{2k}$, because there are at most $2k$ integers each bounded by $n_k+m_k$. Using \eqref{estimate2}, we have for large $k$
\begin{equation}\label{estimate4}
    (n_k+m_k)^{2k}\le (n_k+m_k)^{\frac{2}{C_1}\log(n_k+m_k)}\le2^{(n_k+m_k)\vare}.
\end{equation}
\\ \hspace*{1em} 
Now combine \eqref{Ecover}, \eqref{cover1},  \eqref{etimate3} and \eqref{estimate4}. The $(\tilde{s}+2\vare)$-dimensional Hausdorff measure of $E(\al,\be)$ satisfies
\begin{align*}
    \mathcal{H}^{\tilde{s}+2\vare}(E(\al,\be))&
    \le \liminf_{m\to\infty}\sum_{k=m}^{\infty}\sum_{n_k+m_k\ge e^{C_1k}}^{\infty}(n_k+m_k)^{2k}\,C^k\left(\frac{1}{2}\right)^{2(n_k+m_k)\vare}\\
    &\le \liminf_{m\to\infty}\sum_{k=m}^{\infty}\sum_{n_k+m_k\ge e^{C_1k}}^{\infty}C^k\left(\frac{1}{2}\right)^{(n_k+m_k)\vare}\\
    &\le\frac{C^k}{1-(\frac{1}{2})^{\vare}}\liminf_{m\to\infty}\sum_{k=m}^{\infty}\left(\frac{1}{2}\right)^{\vare e^{C_1k}}<+\infty.
\end{align*}
Thus, for $0<\al<\frac{\be}{1+\be}<\be<1$, we have
$$\dim_{\mathrm{H}}E(\al,\be)\le \tilde{s}+2\vare = s\left(\frac{\eta}{1-\eta}-\vare\right)+2\vare.$$
Since $\vare>0$ is arbitrary and $s(u)$ is continuous, letting $\vare\to 0$ yields
$$\dim_{\mathrm{H}}E(\al,\be)\le s\left(\frac{\eta}{1-\eta}\right)=s\left(\frac{\be^2(1-\al)}{(1-\be)[\be-\al(1+\be)]}\right).$$
\\ \hspace*{1em} 
For $0<\al=\frac{\be}{1+\be}<\be<1$, we keep the definitions of $\mathcal{A}$, $\mathcal{W}_k, \mathcal{O}_k(\{i_j\}_{j=1}^k)$ unchanged and only modify $\mathcal{C}_k$ by replacing \eqref{mk1} with \eqref{mk2}. Set $\tilde{s}=s(M)$. As $M\to\infty$, we have $s(M)\to0$. Letting $\vare\to0$ and $M\to\infty$, we obtain
$$\dim_{\mathrm{H}}E(\al,\be)\le s(M)+2\vare\to 0.$$

\end{description}

\section{Proof of Theorem \ref{mainthm}: Lower bound}
\label{sec:lower-bound}

Recall that $\dim_{\mathrm{H}}E(\alpha,\beta)=0$ when $0<\frac{\beta}{1+\beta}\le\alpha\le\beta<1$. 
Thus it suffices to prove the lower bound for the case $0\le\alpha<\frac{\beta}{1+\beta}<\beta<1$.

\subsection{Construct subset of $E(\alpha,\beta)$}
Choose two sequences of integers $\{n_k\}_{k\ge 1}$ and $\{m_k\}_{k\ge 1}$ as follows.
\begin{itemize}
    \item If $\al>0$, set
    \[
    n_k=\left\lfloor\left(\frac{\beta(1-\alpha)}{\alpha(1-\beta)}\right)^k\right\rfloor+4k,\qquad 
    m_k=\left\lfloor\frac{\beta}{1-\beta}\left(\frac{\beta(1-\alpha)}{\alpha(1-\beta)}\right)^k\right\rfloor+2;
   \]
   \item If $\al=0$, set
   \[
   n_k=\left\lfloor\frac{2^{k!}}{(1-\beta)^k}\right\rfloor+4k,\qquad 
   m_k=\left\lfloor\frac{\beta}{1-\beta}\frac{2^{k!}}{(1-\beta)^k}\right\rfloor+2.
  \]
\end{itemize}
In both cases one checks that $m_k$ is increasing and satisfies the condition
\begin{equation}\label{mk3}
    2\le m_k<n_{k+1}-n_k\quad\text{for all }k\ge 1,\qquad \lim_{k\to\infty}m_k=\infty.
\end{equation}

For any $k\ge 1$, let
\[
p_k=\max\left\{p\in\mathbb{N}:p\le \frac{n_{k+1}-n_k}{m_k}<p+1\right\}.
\]
Define
\[
n'_1=n_1+1\quad\text{and}\quad n'_k=n_k+\sum_{i=1}^{k-1}(p_i+1).
\]
By \eqref{mk3}, we have
\begin{equation}\label{4.1}
\lim_{k\to\infty}\frac{n'_k}{n_k}=1+\lim_{k\to\infty}\frac{p_k+1}{n_{k+1}-n_k}=1.
\end{equation}

Let $M\ge 3$ be an integer and set
\begin{equation}\label{4.2}
   G(M)=\left\{x\in(0,1]:d_n(x)
\begin{cases}
    =2 & \text{if } n'_k+1\le n\le n'_k+m_k\ \text{for all}\ k\ge 1;\\
    =2M & \text{if } n=n'_k,\ n'_k+pm_k+1\ \text{for all}\ 1\le p\le p_k,\ k\ge 1;\\
    \in[2,M] & \text{otherwise}.
\end{cases}
\right\} 
\end{equation}

\begin{proposition}\label{Gcover}
$G(M)\subset E(\alpha,\beta)$.
\end{proposition}
\begin{proof}
Fix $x\in G(M)$. 
Assume $n'_k+m_k\le n<n'_{k+1}+m_{k+1}$.  
Then we have
\[
\ell_n(x)=\begin{cases}
m_k, & \text{if}~~n'_k+m_k\le n\le n'_{k+1}+m_k,\\
n-n'_{k+1}, & \text{if}~~n'_{k+1}+m_k<n<n'_{k+1}+m_{k+1}.
\end{cases}
\]
Consequently,
\[
\liminf_{n\to\infty}\frac{\ell_n(x)}{n}
= \liminf_{k\to\infty}\frac{m_k}{n'_{k+1}+m_k}
= \liminf_{k\to\infty}\frac{m_k}{n_{k+1}+m_k}= \alpha,
\]
\[
\limsup_{n\to\infty}\frac{\ell_n(x)}{n}
= \limsup_{k\to\infty}\frac{m_k}{n'_k+m_k}
= \limsup_{k\to\infty}\frac{m_k}{n_k+m_k}= \beta,
\]
where we use \eqref{4.1}. Hence $x\in E(\alpha,\beta)$.
\end{proof}

Let $\mathcal{J}$ be the set of deleted positions $\{n'_k, n'_k+pm_k+1:k\ge1,1\le p\le p_k\}$. Define $f:G(M)\to f(G(M))$ 
by eliminating the digits in $\mathcal{J}$. That is
\begin{equation}\label{4.3}
    f(G(M))=\left\{x\in(0,1]:d_n(x)
\begin{cases}
    =2 & \text{if } n_k+1\le n\le n_k+m_k\ \text{for all}\ k\ge 1  ;\\        
    \in[2,M] & \text{if } 1\le n\le n_1~\text{and}~ n_k+m_k+1\le n\le n_{k+1}\ \text{for all}\ k\ge 1.
\end{cases}
\right\}
\end{equation}

Write $t(n)=\#\{j\le n:j\in \mathcal{J}\}$. We have 
\begin{equation}\label{lem:t(n)}
    \displaystyle\lim_{n\to\infty}\frac{t(n)}{n}=0.
\end{equation}
Indeed, assume $n'_k\le n<n'_{k+1}$ for some $k\ge 1$. Then
\begin{align*}
    \limsup_{n\to\infty}\frac{t(n)}{n}&\le \limsup_{n\to\infty}\frac{\sum_{i=1}^{k-1}(p_i+1)+\frac{n-n'_k}{m_k}+1}{n}\\
    &\le\limsup_{n\to\infty}\frac{\sum_{i=1}^{k-1}(p_i+1)}{n'_k}+\limsup_{n\to\infty}\frac{1}{m_k}+\limsup_{n\to\infty}\frac{1}{n'_k}\\
    &\le\limsup_{n\to\infty}\frac{p_k+1}{n'_{k+1}-n'_k}\le \limsup_{n\to\infty}\frac{p_k+1}{n_{k+1}-n_k}=0.
\end{align*}

\begin{proposition}\label{lem:Holder}
For any $\varepsilon>0$, the map $f$ is $\frac{1}{1+\varepsilon}$-H\"older continuous on 
$G(M)$, that is, there exists a constant $C_2=C_2(M)>0$ such that for any 
$x,y\in G(M)$,
\begin{equation}\label{4.4}
    |f(x)-f(y)|\le C_2|x-y|^{\frac{1}{1+\varepsilon}}.
\end{equation}
\end{proposition}
\begin{proof}
Take $x\neq y$ and let $n$ be the largest integer with $d_i(x)=d_i(y)=d_i$ for all $1\le i\le n$.
Then $d_{n+1}(x)\neq d_{n+1}(y)$ and both are in $\{2,\dots,M\}$. Without loss with generality, we assume $d_{n+1}(x)<d_{n+1}(y)$.
The interval $I_{n+2}(d_1,\dots,d_n,d_{n+1}(x),2M+1)$ lies between $x$ and $y$, so
\begin{align*}
    |x-y|&\ge |I_{n+2}(d_1,\dots,d_n,d_{n+1}(x),2M+1)|
\ge |I_{n+2}(d_1,\dots,d_n,M,2M+1)|\\
&\ge \frac{1}{d_1(d_1-1)\cdots d_n(d_n-1)}\cdot\frac{1}{M(M-1)}\cdot\frac{1}{(2M+1)2M}\\
&\ge \frac{1}{(M-1)^4}|I_n(d_1,\dots,d_n)|.
\end{align*}

Let $(\overline{d_1,\dots,d_n})$ be the word obtained by deleting from $(d_1,\dots,d_n)$ those positions 
belonging to $\mathcal{J}$. Then $f(x)$ and $f(y)$ both lie in the cylinder 
$I_{n-t(n)}(\overline{d_1,\dots,d_n})$, hence
\begin{align*}
    |f(x)-f(y)|&\le |I_{n-t(n)}(\overline{d_1,\dots,d_n})|=\frac{1}{\prod_{j\notin\mathcal{J},\,j\le n} d_j(d_j-1)}\\
    &= |I_n(d_1,\dots,d_n)|\cdot\prod_{j\in\mathcal{J},\,j\le n} d_j(d_j-1)\\
&\le |I_n(d_1,\dots,d_n)|(2M(2M-1))^{t(n)}.
\end{align*}

By \eqref{lem:t(n)}, for any $\varepsilon>0$ there exists $N_2$ such that for $n\ge N_2$,
\[
t(n)\log(2M(2M-1))\le \varepsilon(n-t(n))\log 2.
\]
Hence,
\[
|I_{n-t(n)}(\overline{d_1,\dots,d_n})|
\le |I_n(d_1,\dots,d_n)|\cdot (2M(2M-1))^{t(n)}
\le |I_n(d_1,\dots,d_n)|\cdot 2^{\varepsilon(n-t(n))}.
\]

Since $|I_n(d_1,\dots,d_n)|\le 2^{-n}$ and $n-t(n)\ge (1-\delta)n$ for large $n$ (with $\delta$ small),
we have $2^{\varepsilon(n-t(n))}\le |I_n(d_1,\dots,d_n)|^{-\varepsilon(1-\delta)}$.
Choosing $\delta$ so that $\varepsilon(1-\delta)<\varepsilon'<\varepsilon$, we obtain
\[
|f(x)-f(y)|\le |I_n(d_1,\dots,d_n)|^{1-\varepsilon'}
\le \bigl((M-1)^4|x-y|\bigr)^{1-\varepsilon'}
\le C_2|x-y|^{1-\varepsilon'}.
\]
Since $\varepsilon'>0$ can be made arbitrarily small, this proves the $\frac{1}{1+\varepsilon}$-H\"older 
continuity for any $\varepsilon>0$.
\end{proof}

Combining Proposition~\ref{Gcover} and Proposition~\ref{lem:Holder}, using Lemma \ref{lemma2.5}, we obtain
\[
\dim_{\mathrm{H}}E(\alpha,\beta)\ge \dim_{\mathrm{H}}G(M)
\ge \frac{1}{1+\varepsilon}\dim_{\mathrm{H}}f(G(M)).
\]
Letting $\varepsilon\to0$ gives
\begin{equation}\label{4.6}
    \dim_{\mathrm{H}}E(\alpha,\beta)\ge\dim_{\mathrm{H}}f(G(M)).
\end{equation}

It remains to prove the lower bound of the Hausdorff dimension of $f(G(M))$.

\subsection{Symbolic description of $f(G(M))$}

For $n\ge 1$ set
\begin{align*}
    D_n&=\{(d_1,\dots,d_n)\in\Sigma^n:
d_j=2\ \text{ for }\ j\in[n_k+1,n_k+m_k]\cap[1,n]\ \text{and} \ k\ge1,\\\\ 
&~~~~~~~~~~~~~~~~~~~~~~~~~~~~~~~~~~~~~~~~~2\le d_j\le M\text{ otherwise}\},\\
    D&=\bigcup_{n=0}^{\infty}D_n\ (D_0=\emptyset).
\end{align*}
For $(d_1,\dots,d_n)\in D_n$, define
\begin{equation}\label{4.7}
    J_n(d_1,\dots,d_n)=\bigcup_{d_{n+1}:(d_1,\dots,d_{n+1})\in D_{n+1}} I_{n+1}(d_1,\dots,d_{n+1})
\end{equation}
an $n$-th order fundamental interval.
Then
\[
f(G(M))=\bigcap_{n\ge 1}\bigcup_{(d_1,\dots,d_n)\in D_n} J_n(d_1,\dots,d_n).
\]

From Lemma \ref{lemmalength}, the lengths satisfy:
\begin{itemize}
\item If $n_k+m_k\le n<n_{k+1}$, then
\begin{equation}\label{4.8}
    |J_n(d_1,\dots,d_n)|=\frac{M-1}{M}|I_n(d_1,\dots,d_n)|.
\end{equation}
\item If $n_{k+1}\le n<n_{k+1}+m_{k+1}$, then
\begin{equation}\label{4.9}
    |J_n(d_1,\dots,d_n)|=\frac{1}{2}|I_n(d_1,\dots,d_n)|.
\end{equation}
\end{itemize}

\subsection{Measure supported on $f(G_{(\{\widetilde{n_k}\},\{m_k\})}(M))$}

Let $n_0=m_0=0$ and define
\begin{equation}\label{4.10}
    u_k=\frac{m_k}{n_k-(n_{k-1}+m_{k-1})}\quad \text{for all}~ k\ge1.
\end{equation}
A direct computation using the asymptotics of $n_k,m_k$ shows that
\begin{equation}\label{4.11}
    \lim_{k\to\infty}u_k=\zeta:=\frac{\beta^2(1-\alpha)}{(1-\beta)[\beta-\alpha(1+\beta)]}.
\end{equation}
From Lemma \ref{continus}, we have $s_M(u)$ is continuous with respect to $u$. Therefore, for any $\varepsilon>0$ there exists $K_0$ such that for all $k\ge K_0$,
\begin{equation}\label{4.12}
    |s_M(u_k)-s_M(\zeta)|<\frac{\vare\log 2}{\log M(M-1)}=:\tilde{\vare}.
\end{equation}

We define a set function $\mu:\{J(d):d\in D\}\to \mathbb{R}^+$ as follows.
\begin{itemize}
    \item For any $(d_1,\ldots,d_{n_1+m_1})\in D_{n_1+m_1}$, define
$$\mu(J_{n_1+m_1}(d_1,\ldots,d_{n_1+m_1}))=\left(\frac{1}{d_1(d_1-1)\cdots d_{n_1+m_1}(d_{n_1+m_1}-1)}\right)^{s_M(u_1)}.$$
    \item For any
    $(d_1,\ldots,d_n)\in D_n$ with $n<n_1+m_1$, let
$$\mu(J_n(d_1,\ldots,d_n))=\sum_{d_{n+1},\ldots,d_{n_1+m_1}}\mu(J_{n_1+m_1}(d_1,\ldots,d_{n_1+m_1})).$$
    \item Once $\mu(J_{n_k+m_k}(d_1,\ldots,d_{n_k+m_k}))$ has been defined for some $k\ge 1$. Then for any $(d_1,\ldots, \linebreak d_{n_{k+1}+m_{k+1}})\in D_{n_{k+1}+m_{k+1}}$, let
\begin{align}\label{mu1}
    &\nonumber\mu(J_{n_{k+1}+m_{k+1}}(d_1,\ldots,d_{n_{k+1}+m_{k+1}}))\\
    &=\mu(J_{n_k+m_k}(d_1,\ldots,d_{n_k+m_k}))\left(\frac{1}{d_{n_k+m_k+1}\cdots d_{n_{k+1}+m_{k+1}}(d_{n_{k+1}+m_{k+1}}-1)}\right)^{s_M(u_{k+1})}.
\end{align}
   \item For any $(d_1,\ldots,d_n)\in D_n$ with $n_k+m_k<n<n_{k+1}+m_{k+1}$, define
   \begin{align}\label{mu2}
       \mu(J_n(d_1,\ldots,d_n))=\sum_{d_{n+1},\ldots,d_{n_{k+1}+m_{k+1}}}\mu(J_{n_{k+1}+m_{k+1}}(d_1,\ldots,d_{n_{k+1}+m_{k+1}})).
   \end{align}
\end{itemize}

Because $s_M(u_k)$ satisfies
\[
\sum_{t=2}^M\left(\frac{1}{2^{u_k}t(t-1)}\right)^{s_M(u_k)}=1,
\]
the set function satisfies the Kolmogorov’s consistency condition. So it can be uniquely extended to a probability measure supported on $f(G(M))$, still denoted by $\mu$. 

From the construction we have, for any $k\ge1$ and $(d_1,\dots,d_{n_k+m_k})\in D_{n_k+m_k}$,
\begin{align}\label{4.14}
&\mu(J_{n_k+m_k}(d_1,\dots,d_{n_k+m_k}))\nonumber\\
&\quad =\prod_{j=1}^k\left(\frac{1}{d_{n_{j-1}+m_{j-1}+1}(d_{n_{j-1}+m_{j-1}+1}-1)\cdots
d_{n_j+m_j}(d_{n_j+m_j}-1)}\right)^{s_M(u_j)}.
\end{align}

\begin{proposition}\label{prop:measure-bound}
Let $n>n_{K_0}+m_{K_0}$, where $K_0$ satisfies \eqref{4.12}. For any $\varepsilon>0$ there exists a constant $C_3=C_3(M)$ such that for every 
$(d_1,\dots,d_n)\in D_n$,
\begin{equation}\label{4.15}
    \mu(J_n(d_1,\dots,d_n))\le C_4|J_n(d_1,\dots,d_n)|^{s_M(\zeta)-\varepsilon}.
\end{equation}
\end{proposition}
\begin{proof}
We distinguish three cases. We take the notation $\Sigma_M=\{2,\ldots,M\}$.

\noindent\textbf{Case 1:} $n=n_k+m_k$.
Using \eqref{4.12} and \eqref{4.14},
\begin{align}
        &\nonumber\mu(J_{n_k+m_k}(d_1,\ldots,d_{n_k+m_k}))\\
        &\nonumber\le \prod_{j=K_0}^{k}\left(\frac{1}{d_{n_{j-1}+m_{j-1}+1}(d_{n_{j-1}+m_{j-1}+1}-1)\cdots d_{n_{j}+m_{j}}(d_{n_{j}+m_{j}}-1)}\right)^{s_M(\zeta)-\vare}\\
        &\nonumber= \prod_{j=1}^{K_0-1}\left(d_{n_{j-1}+m_{j-1}+1}(d_{n_{j-1}+m_{j-1}+1}-1)\cdots d_{n_{j}+m_{j}}(d_{n_{j}+m_{j}}-1)\right)^{s_M(\zeta)-\vare} \\
        &\nonumber~~~~~~~~~\cdot\prod_{j=1}^{k}\left(\frac{1}{d_{n_{j-1}+m_{j-1}+1}(d_{n_{j-1}+m_{j-1}+1}-1)\cdots d_{n_{j}+m_{j}}(d_{n_{j}+m_{j}}-1)}\right)^{s_M(\zeta)-\vare}\\
        &\le M^{2(n_{K_0}+m_{K_0})}|I_{n_k+m_k}(d_1,\ldots,d_{n_k+m_k})|^{s_M(\zeta)-\vare} \label{case1mu}\\
        &\nonumber\le M^{2(n_{K_0}+m_{K_0}+1)}|J_{n_k+m_k}(d_1,\ldots,d_{n_k+m_k})|^{s_M(\zeta)-\vare}.
    \end{align}
    where the last inequality comes from \eqref{4.8}.

\noindent\textbf{Case 2:} $n_k+m_k<n<n_{k+1}$.
Recall \eqref{mu1}, \eqref{mu2} and \eqref{case1mu}, we obtain
   \begin{align*}
       &\mu(J_n(d_1,\ldots,d_n))\\
       &=\mu(J_{n_k+m_k}(d_1,\ldots,d_{n_k+m_k}))\left(\frac{1}{d_{n_k+m_k+1}(d_{n_k+m_k+1}-1)\cdots d_n(d_n-1)}\right)^{s_M(u_{k+1})}\\
       &~~~~~~\cdot\sum_{\substack{ d_{n+1},\ldots,d_{n_{k+1}}\in\Sigma_M \\ d_{n_{k+1}+1}=\cdots=d_{n_{k+1}+m_{k+1}}=2}} \left(\frac{1}{d_{n+1}(d_{n+1}-1)\cdots d_{n_{k+1}+m_{k+1}}(d_{n_{k+1}+m_{k+1}}-1)}\right)^{s_M(u_{k+1})}\\
       &\le M^{2(n_{K_0}+m_{K_0})}|I_{n_k+m_k}(d_1,\ldots,d_{n_k+m_k})|^{s_M(\zeta)-\vare}\left(\frac{1}{d_{n_k+m_k+1}(d_{n_k+m_k+1}-1)\cdots d_n(d_n-1)}\right)^{s_M(\zeta)-\vare}\\
       &~~~~~~\cdot\sum_{\substack{ d_{n+1},\ldots,d_{n_{k+1}}\in\Sigma_M \\ d_{n_{k+1}+1}=\cdots=d_{n_{k+1}+m_{k+1}}=2}} \left(\frac{1}{d_{n+1}(d_{n+1}-1)\cdots d_{n_{k+1}+m_{k+1}}(d_{n_{k+1}+m_{k+1}}-1)}\right)^{s_M(u_{k+1})}\\
       &\le M^{2(n_{K_0}+m_{K_0})}|I_{n}(d_1,\ldots,d_{n})|^{s_M(\zeta)-\vare}\\
       &~~~~~~\cdot\sum_{\substack{ d_{n+1},\ldots,d_{n_{k+1}}\in\Sigma_M \\ d_{n_{k+1}+1}=\cdots=d_{n_{k+1}+m_{k+1}}=2}} \left(\frac{1}{d_{n+1}(d_{n+1}-1)\cdots d_{n_{k+1}+m_{k+1}}(d_{n_{k+1}+m_{k+1}}-1)}\right)^{s_M(u_{k+1})}.
   \end{align*}

Notice that
\begin{align*}  
    1&=\sum_{\sigma_1,\ldots,\sigma_{n-(n_k+m_k)}\in \Sigma_M} \left(\frac{1}{\sigma_1(\sigma_1-1)\ldots \sigma_{n-(n_k+m_k)}(\sigma_{n-(n_k+m_k)}-1)}\right)^{s_M(u_{k+1})}\\
    &~~~~\cdot\sum_{\substack{d_{n+1},\ldots,d_{n_{k+1}}\in \Sigma_M\\ d_{n_{k+1}+1}=\cdots=d_{n_{k+1}+m_{k+1}}=2}} \left(\frac{1}{d_{n+1}(d_{n+1}-1)\cdots d_{n_{k+1}+m_{k+1}}(d_{n_{k+1}+m_{k+1}}-1)}\right)^{s_M(u_{k+1})}.
\end{align*}

By \eqref{4.12} and the definition of $S_M(\zeta)$, we have
\begin{align*}
    &\sum_{\sigma_1,\ldots,\sigma_{n-(n_k+m_k)}\in\Sigma_M} \left(\frac{1}{\sigma_1(\sigma_1-1)\ldots \sigma_{n-(n_k+m_k)}(\sigma_{n-(n_k+m_k)}-1)}\right)^{s_M(u_{k+1})}\\
    &\ge \sum_{\substack{\sigma_{(n-(n_k+m_k))\frac{1}{1+\zeta}+1}=\ldots=\sigma_{n-(n_k+m_k)}=2\\\sigma_1,\ldots,\sigma_{(n-(n_k+m_k))\frac{1}{1+\zeta}}\in\Sigma_M}}\left(\frac{1}{\sigma_1(\sigma_1-1)\ldots \sigma_{n-(n_k+m_k)}(\sigma_{(n-(n_k+m_k))\frac{1}{1+\zeta}}-1)2\cdots2}\right)^{s_M(\zeta)+\tilde{\vare}}\\
    &\ge \left(\frac{1}{M(M-1)}\right)^{(n-(n_k+m_k))\tilde{\vare}}\ge \left(\frac{1}{M(M-1)}\right)^{n\tilde{\vare}}.
\end{align*}
Then we get
\begin{align*}
&(M(M-1))^{n\tilde{\vare}}\ge\\
    &~~~~~~~~\sum_{\substack{d_{n+1},\ldots,d_{n_{k+1}}\in \Sigma_M
\\d_{n_{k+1}+1}=\cdots=d_{n_{k+1}+m_{k+1}}=2}} \left(\frac{1}{d_{n+1}(d_{n+1}-1)\cdots d_{n_{k+1}+m_{k+1}}(d_{n_{k+1}+m_{k+1}}-1)}\right)^{s_M(u_{k+1})}.
\end{align*}
By the definition of \eqref{4.12}, we have
\begin{align*}
    \mu(J_n(d_1,\ldots,d_n))&\le M^{2(n_{K_0}+m_{K_0})}(M(M-1))^{n\tilde{\vare}}|I_n(d_1,\ldots,d_n)|^{s_M(\zeta)-\vare}\\
    &\le M^{2(n_{K_0}+m_{K_0})}|I_n(d_1,\ldots,d_n)|^{s_M(\zeta)-2\vare}\\
    &\le M^{2(n_{K_0}+m_{K_0}+1)}|J_n(d_1,\ldots,d_n)|^{s_M(\zeta)-2\vare}.
\end{align*}

\noindent\textbf{Case 3:} If $n_{k+1}\le n<n_{k+1}+m_{k+1}$. Given $d_{n+1}=\cdots=d_{n_{k+1}+m_{k+1}}=2$, we have
\begin{align*}
    \mu(J_n(d_1,\ldots,d_n))&= \mu(J_{n_{k+1}+m_{k+1}}(d_1,\ldots,d_{n_{k+1}+m_{k+1}}))\\
    &\le M^{2(n_{K_0}+m_{K_0}+1)}|J_{n_{k+1}+m_{k+1}}(d_1,\ldots,d_{n_{k+1}+m_{k+1}})|^{s_M(\zeta)-\vare}\\
    &\le M^{2(n_{K_0}+m_{K_0}+1)}|J_n(d_1,\ldots,d_n)|^{s_M(\zeta)-\vare}.
\end{align*}
Take $C_3=M^{2(n_{K_0}+m_{K_0}+1)}$. The proof is completed.

\end{proof}

\subsection{Gap between fundamental intervals}

For $(d_1,\dots,d_n)\in D_n$ let $g_n(d_1,\dots,d_n)$ be the distance from $J_n(d_1,\dots,d_n)$ 
to the nearest distinct $n$-th order fundamental interval.

\begin{proposition}\label{lem:gap}
There exists a constant $C_4=C_4(M)>0$ such that for every $(d_1,\dots,d_n)\in D_n$,
\begin{equation}\label{4.16}
    g_n(d_1,\dots,d_n)\ge C_4|J_n(d_1,\dots,d_n)|.
\end{equation}
\end{proposition}
\begin{proof}
In this subsection, we focus on estimating the gaps between disjoint fundamental intervals, as defined in \eqref{4.7}, of the same order. For each $(d_1,\ldots,d_n)\in D_n$, we denote the distance between $J_n:=J_n(d_1,\ldots,d_n)$ and the nearest fundamental interval of the same order on its right (left, respectively) as $g^r(d_1,\ldots,d_n)$ ($g^{\ell}(d_1,\ldots,d_n)$, respectively). We define the gap as
$$g(d_1,\ldots,d_n):=\min\{g^r(d_1,\ldots,d_n),g^{\ell}(d_1,\ldots,d_n)\}.$$
Let $J_n^r:=J_n(d_1^r,\ldots,d_n^r)$ and $J_n^{\ell}:=J_n(d_1^{\ell},\ldots,d_n^{\ell})$ be the adjoining fundamental intervals (if exists) lying on the left and right of $J_n(d_1,\ldots,d_n)$.

\textbf{Case 1:} If $n_k+m_k\le n < n_{k+1}$. In this case, we have
\begin{align*}
    J_n&=\bigcup_{2\le d\le M}I_{n+1}(d_1,\ldots,d_n,d),\\
    J_n^{\ell}&=J_n(d_1,\ldots,d_{n-1},d_n+1)=\bigcup_{2\le d\le M}I_{n+1}(d_1,\ldots,d_n+1,d),\\
    J_n^r&=J_n(d_1,\ldots,d_{n-1},d_n-1)=\bigcup_{2\le d\le M}I_{n+1}(d_1,\ldots,d_n-1,d).
\end{align*}
By using \eqref{4.8}, and Lemma \ref{lemmalength}, we have
\begin{align*}
    g^{\ell}(d_1,\ldots,d_n)&\ge \sum_{d>M}|I_{n+1}(d_1,\ldots,d_n,d)|\\
    &=\sum_{d>M}\frac{1}{d_1(d_1-1)\cdots d_n(d_n-1)d(d-1)}\\
    &=\frac{1}{M}|I_{n}(d_1,\ldots,d_n)|=\frac{1}{M-1}|J_{n},(d_1,\ldots,d_n)|,
\end{align*}
and
$$g^{r}(d_1,\ldots,d_n)\ge \sum_{d>M}|I_{n+1}(d_1,\ldots,d_n-1,d)|\ge \frac{1}{M-1}|J_{n},(d_1,\ldots,d_n)|.$$
Thus, we obtain 
\begin{equation}\label{gn1}
    g_n(d_1,\ldots,d_n)\ge \frac{1}{M-1}|J_{n}(d_1,\ldots,d_n)|.
\end{equation}

\textbf{Case 2:} If $n_{k+1}\le n< n_{k+1}+m_{k+1}$. Observe that
$$J_n=J_n(d_1,\ldots,d_{n_{k+1}},\underbrace{2,\ldots,2}_{n-n_{k+1}}),\qquad J_n^{\ell}=J_n(d_1^{\ell},\ldots,d_{n_{k+1}}^{\ell},\underbrace{2,\ldots,2}_{n-n_{k+1}}).$$
We get $J_{n_{k+1}}^{\ell}$ is the left adjoining fundamental interval $J_{n_{k+1}}$, and the distance between $J_n$ and $J_n^{\ell}$ is larger than the one between $J_{n_{k+1}}^{\ell}$  and $J_{n_{k+1}}$. Then we get
\begin{align*}
    g_n^{\ell}&(d_1,\ldots,d_{n_k+m_k},2,\ldots,2)\ge g_{n_k+m_k}^{\ell}(d_1,\ldots,d_{n_k+m_k})\\
    &\ge \frac{1}{M-1}|J_{n_k+m_k}(d_1,\ldots,d_{n_k+m_k})|\\
    &\ge \frac{1}{M-1}|J_n(d_1,\ldots,d_{n_k+m_k},2,\ldots,2)|.
\end{align*}
Using the same method, we obtain
$$g_n^{r}(d_1,\ldots,d_{n_k+m_k},2,\ldots,2)\ge\frac{1}{M-1}|J_n(d_1,\ldots,d_{n_k+m_k},2,\ldots,2)|.$$
Thus, it follows that
\begin{equation}\label{gn2}
    g_n(d_1,\ldots,d_n)\ge \frac{1}{M-1}|J_{n}(d_1,\ldots,d_n)|.
\end{equation}
Take $C_4=\frac{1}{M-1}$, the proof is finished.
\end{proof}

\subsection{Completion of the lower bound}

\begin{proposition}\label{prop:ball-measure}
For any $\varepsilon>0$ there exists $\tilde{C}=\tilde{C}(M,\varepsilon)$ such that for every 
$x\in f(G(M))$ and sufficiently small $r>0$,
\begin{equation}\label{4.17}
    \mu(B(x,r))\le \tilde{C}\, r^{s_M(\zeta)-\varepsilon}.
\end{equation}
\end{proposition}
\begin{proof}
Let $r_0=\min\limits_{1\le j\le n_{K_0}+m_{K_0}}\min\limits_{(d_1,\ldots,d_j)\in D_j}g(d_1,\ldots,d_j)>0$.
Take $x\in f(G(M))$ and $0<r<r_0$.
There exists a unique sequence $d_1, d_2,\ldots$ satisfying $x\in J_k(d_1,\ldots,d_k)$ for all $k\ge 1$.
Choose $n$ such that
\[
g_{n+1}(d_1,\dots,d_{n+1})\le r < g_n(d_1,\dots,d_n).
\]
By Proposition~\ref{lem:gap}, the ball $B(x,r)$ can intersect only one $n$-th order cylinder, namely $J_n(d_1,\dots,d_n)$. From Proposition~\ref{prop:measure-bound}, we deduce that
\[
\mu(B(x,r))\le \mu(J_n(d_1,\dots,d_n))\le C_4|J_n(d_1,\dots,d_n)|^{s_M(\zeta)-\varepsilon}.
\]

By \eqref{4.8} and \eqref{4.9}, we see that
\begin{itemize}
    \item when $n_k+m_k\le n<n_{k+1}-1$ and $n_{k+1}\le n <n_{k+1}+m_{k+1}-1$, we have
    $$\frac{|J_n(d_1,\ldots,d_n)|}{|J_{n+1}(d_1,\ldots,d_{n+1})|}=\frac{|I_n(d_1,\ldots,d_n)|}{|I_{n+1}(d_1,\ldots,d_{n+1})|}=d_{n+1}(d_{n+1}-1),$$
    \item when $n=n_{k+1}-1$, we have
    $$\frac{|J_n(d_1,\ldots,d_n)|}{|J_{n+1}(d_1,\ldots,d_{n+1})|}=\frac{M}{2(M-1)}d_{n+1}(d_{n+1}-1),$$
    \item when $n=n_{k+1}+m_{k+1}-1$, we have
    $$\frac{|J_n(d_1,\ldots,d_n)|}{|J_{n+1}(d_1,\ldots,d_{n+1})|}=\frac{2(M-1)}{M}d_{n+1}(d_{n+1}-1).$$
\end{itemize}
Since $d_{n+1}\le M$, it holds that
    $$\frac{|J_n(d_1,\ldots,d_n)|}{|J_{n+1}(d_1,\ldots,d_{n+1})|}\le 2(M-1)^2$$
    for any $n_k+m_k\le n<n_{k+1}+m_{k+1}$ with $k\ge K_0$. 
    
Therefore,
    \begin{align*}
        \mu(B(x,r))&\le C_3|J_{n}(d_1,\ldots,d_{n})|^{s_M(\zeta)-\vare}\\
        &\le C_3 2(M-1)^2|J_{n+1}(d_1,\ldots,d_{n+1})|^{s_M(\zeta)-\vare}\\
        &\le  2(M-1)^2 C_3C_4^{-(s_M(\zeta)-\vare)}g_{n+1}(d_1,\ldots,d_{n+1})^{s_M(\zeta)-\vare}\\
        &\le 2(M-1)^2 C_3C_4^{-(s_M(\zeta)-\vare)}r^{s_M(\zeta)-\vare}.
    \end{align*}

\end{proof}

Applying the mass distribution principle (see e.g. \cite{kf90}) to Proposition~\ref{prop:ball-measure},
we obtain
\[
\dim_{\mathrm{H}}f(G(M))\ge s_M(\zeta)-\varepsilon.
\]
Letting $\varepsilon\to0$ gives $\dim_{\mathrm{H}}f(G(M))\ge s_M(\zeta)$.

Finally, by \eqref{4.6},
\[
\dim_{\mathrm{H}}E(\alpha,\beta)\ge s_M(\zeta).
\]
Clearly $s_M(u)\le s(u)$ and $\lim\limits_{M\to\infty}s_M(u)= s(u)$. Letting $M\to\infty$ yields
\[
\dim_{\mathrm{H}}E(\alpha,\beta)\ge s(\zeta)
= s\!\left(\frac{\beta^2(1-\alpha)}{(1-\beta)[\beta-\alpha(1+\beta)]}\right).
\]

This completes the proof of the lower bound. Together with the upper bound established in Section~\ref{sec:upper-bound},
Theorem~\ref{mainthm} is proved.

\end{document}